\documentclass[a4paper,11pt]{article}
\usepackage[utf8]{inputenc}
\usepackage{amsmath}
\usepackage{amssymb}
\usepackage{amsthm}
\usepackage{dsfont}
\usepackage{hyperref}
\usepackage{graphicx}
\usepackage{enumerate}
\usepackage{booktabs}
\usepackage{arydshln}
\usepackage{mathtools}
\usepackage[round]{natbib}
\usepackage[affil-it]{authblk}
\usepackage[misc]{ifsym}

\newcommand{\Pro}{\mathbb{P}}
\newcommand{\R}{\mathbb{R}}
\newcommand{\E}{\mathbb{E}}
\newcommand{\U}{\mathcal{U}}

\newcommand{\A}{\mathcal{A}}
\newcommand{\G}{\mathcal{G}}

\DeclareMathOperator*{\argmin}{arg\,min}

\newcommand\blfootnote[1]{%
  \begingroup
  \renewcommand\thefootnote{}\footnotetext{#1}%
  \addtocounter{footnote}{-1}%
  \endgroup
}

\newtheorem{theorem}{Theorem}[section]

\newtheorem{lemma}[theorem]{Lemma}

\newtheorem{remark}[theorem]{Remark}

\begin{document}

\title{Optimal Reinsurance for Gerber-Shiu Functions in the Cram\'{e}r-Lundberg Model}
%\subtitle{none}
%\titlerunning{Optimal reinsurance for Gerber Shiu functions}
\author{\normalsize M. Preischl and S. Thonhauser\thanks{The authors
are supported by the Austrian Science Fund (FWF) Project F5510 (part of the Special Research
Program (SFB) \textquotedblleft Quasi-Monte Carlo Methods: Theory and Applications\textquotedblright).
}}

%\author{Michael Preischl \and Stefan Thonhauser}
% Use \authorrunning{Short Author} for an abbreviated list of authors

% \institute{
%  Michael Preischl(\Letter)
%  \at Institute of Analysis and Number Theory, Graz University of Technology,
% Steyrergasse 30/II, 8010 Graz, Austria \\%Graz University of Technology, Kopernikusgasse 24, 8010 Graz \\
%  \email{preischl@math.tugraz.at;}
%  Stefan Thonhauser
%  \at Institute for Statistics, Graz University of Technology,
% Steyrergasse 30/III, 8010 Graz, Austria \\
% \email{stefan.thonhauser@math.tugraz.at;}}
%  %
%  \and
%  %
% }
% See the above example for multiple authors with a common address,
% and one author with multiple addresses
% Mark the corresponding author with \Letter

\maketitle
% List each author here on a separate line
% "Lastname, Firstname Initials." as they should appear in the index
\index{Preischl, Michael}
\index{Thonhauser, Stefan}
\abstract{
Complementing existing results on minimal ruin probabilities, we minimize expected discounted penalty functions (or Gerber-Shiu functions) in a Cram\'{e}r-Lundberg model by choosing optimal reinsurance. Reinsurance strategies are modelled as time dependant control functions, which leads to a setting from the theory of optimal stochastic control and ultimately to the problem's Hamilton-Jacobi-Bellman equation. We show existence and uniqueness of the solution found by this method and provide numerical examples involving light and heavy tailed claims and also give a remark on the asymptotics. 
}
\section{Introduction and Preliminaries}
\subsection{Motivation}
The problem of choosing an optimal reinsurance contract has been a very active field inside actuarial mathematics for several years and numerous different frameworks have been considered in this context. The earlier works on this topic were inspired by \cite{waters1983some} where the idea is to maximize the adjustment coefficient to achieve the fastest decay rate for the ruin probability with increasing initial capital. While this approach is focused on the asymptotic behaviour and therefore results in a static reinsurance strategy, \cite{schmidli2001optimal, schmidli2002minimizing}, \cite{hipp2003optimal} and \cite{hipp2010optimal} considered dynamic control strategies, so the reinsurance policy can adapt to the evolution of the reserve process. A collection of results on optimal dynamic reinsurance can be found in \cite{schmidli2007stochastic}. Like the papers cited above, most authors working on dynamic reinsurance take the perspective of optimal stochastic control. A comprehensive summary of these methods in insurance mathematics is provided by \cite{azcue2014stochastic}.\\
Many different approaches can be made, depending on whether or not capital injections are considered, a diffusion term is added to the risk process and also which functional is to be optimized. For the latter question, the most popular choice is the ruin probability but other functionals are thinkable and interesting. For example \cite{azcue2005optimal} and \cite{cani2017optimal} ask for the strategy maximizing a dividend payoff and it is shown that results are qualitatively different from optimal strategies for minimizing the probability of ruin. In our manuscript, we will consider a quite general selection of functionals combined in the notion of discounted penalty functions, a concept that is widely used in many branches of insurance mathematics. 
\subsection{The model}
We consider a risk reserve process $(X_t)_{t\geq0}$ in the classical Cram\'{e}r-Lundberg model. That is, starting from some initial value $x$, the reserve process evolves over time subject to premium income and claim occurence. The claim arrivals are given by a Poisson process with intensity $\lambda$, i.e. there are $\lambda$ claims to be expected per unit time (equivalently, the expected inter claim time is $\frac1\lambda$). The claim heights are independent of this Poisson process and follow some continuous distribution $F_Y$ on $(0,\infty)$. Although not strictly necessary, we will in general assume that $F_Y$ has a density $f_Y$.\\
We assume that reinsurance can be obtained in the form of a control function $u$ in the following sense:\\
At each point in time $t$, a control parameter $u$ is chosen from a set $U$ (e.g. $U=[0,1]$). The map $u_t:\R^+\to U$ is called the reinsurance strategy and by $\mathcal{U}$ we denote the set of processes on $U$ that are previsible with respect to $\mathcal{F}^X$, the sigma algebra generated by the process $X_t$. The functions in $\U$ are called \textit{admissible control strategies}.\\
%For each surplus $X_t$, a control parameter $u_t=u(X_t)$ is chosen from a set $U$ (e.g. $U=[0,1]$). The map $u:\R\to U$ is called the reinsurance strategy and by $\mathcal{U}\subset \mathcal{C}(\R,U)$ we denote the set of \textit{admissible control strategies}.\\
The effect of the reinsurance is modelled by the \textit{retention function} $r$: If a claim of height $y$ is encountered at time $t$, only the part $r(y,u_t)$ is to be paid by the insurer, the rest of the cost is transferred to the reinsurance company. Throughout the paper, we assume $r$ to be monotone in $y$ and continuous in $u$. Of course, reinsurance is not for free and so the reinsurance strategy also influences the reinsurance premiums and thus ultimately the premium income of the first insurer (in the following also called the cedent). Therefore, the premium rate at time $t$ is calculated as
\[
 c(u_t)=c-p(u_t),
\]
where $c$ denotes the cedent's premiums without reinsurance and $p(u_t)$ is the reinsurer's premium. These premiums can be calculated in several ways, including the \textit{expectation principle}, the \textit{variance principle} and the \textit{exponential principle} as some of the most popular ones.
%will be denoted by $c(u(X_t))$.
Throughout this article, we want to assume that the reinsurance premium is in relation higher than the cedent's premium. So buying full reinsurance will result in a negative premium rate. \\
% rather make everything state dependent (i.e. on x) than time dependent???
$ $\\
Combining these assumptions, we define the process $X_t^u$ controlled by the strategy $u\in\mathcal{U}$:
\[
 X_t^u=x+\int_0^t c(u_s)\,ds-\sum_{i=1}^{N_t}r(Y_i,u_{T_i}).
\]
Here, and in the rest of the paper, $N_t$ denotes the number of claims up to time $t$ and $T_i$ resp. $Y_i$ denotes the time resp. the height of the $i$-th claim.\\
$ $\\
Let $\tau_x^u:=\inf\{t\geq0:X_t^u\leq0|X_0^u=x\}$ denote the \textit{time of ruin}, i.e. the first point in time at which $X^u_t$ becomes negative. For convenience, we freeze the process after the ruin event, that is $X_t^u=X_{\tau_x^u}^u$ for all $t>\tau_x^u$. Following \cite{gerber1998time}, we are interested in discounted penalty functions (or \textit{Gerber-Shiu functions}) of the following form
\[
\Phi^u(x):=\E_x\left[e^{-\delta\tau_x^u}w(X_{\tau_x^u-}^u,|X_{\tau_x^u}^u|)\mathds{1}_{\tau_x^u<\infty}\right].
\]
Here, $X_{\tau_x^u-}^u$ is called \textit{surplus prior to ruin}, $|X_{\tau_x^u}^u|$ is the \textit{deficit at ruin} and $\delta>0$ is a discounting factor. Throughout this article, we demand that $w:\R^+\times\R^+\to\R^+$ is a continuous function. Given that we want to minimize the penalty, we are left with finding
\[
 V(x):=\inf_{u\in\mathcal{U}} \Phi^u(x),
\]
for $x>0$. We will also call $V(x)$ the \textit{value function}.
\subsection{Properties of the value function}
To conclude the preliminaries, we want to show two easy but important lemmas, giving monotonicity and, under mild conditions, Lipschitz continuity of $V$.
\begin{lemma}\label{monotonicity}
 $V(x)$ is strictly monotonously decreasing.
\end{lemma}
\begin{proof}
 Let $x>y$. Starting in $x$, buy continuously full reinsurance, resulting in some negative drift $\pi$. Hence, deterministically, after time $\frac{y-x}{\pi}$ the process reaches level $y$. Taking the optimal strategy from there means
\[
 V(x)\leq e^{-\delta\frac{y-x}{\pi}}V(y)<V(y).
\]
\end{proof}
\begin{remark}
 Since Lemma \ref{monotonicity} is a statement about the discounted penalty function of the \textit{optimally} controlled process, it is woth noting that monotonicity does not hold for an arbitrary control strategy.
\end{remark}

\begin{lemma}
 Assume that $w$ (and hence also $\Phi$ and $V$) is bounded by some constant $M$. Then $V(x)$ is Lipschitz continuous.
\end{lemma}
\begin{proof}
For every $x>0$ there is an $\varepsilon$-optimal strategy $\tilde{u}$ which fulfills
 \[
  V(x)\geq \Phi^{\tilde{u}}(x)-\varepsilon.
 \]
Let $y<x$ and let $u_t\equiv u$ be a constant control strategy such that the process has positive drift ($c(u)>\lambda \E[r(Y,u)]$). Now we denote the first hitting time of $x$ from $y$ by $\theta_x:=\inf\{t\geq0:X_t^u\geq x | X_0^u=y\}$ and define a new control strategy $\bar{u}=(\bar{u}_t^y)$ for the process starting in $y$ by $\bar{u}_t=u$ for $0\leq t\leq\theta_x$ and $\bar{u}_t=\tilde{u}_{t-\theta_x}$ for $t\geq\theta_x$.\\
%So let $x>y>0$ and take $\tilde{u}$ to be $\varepsilon$-optimal when starting in $x$. Define the control strategy $\bar{u}$ for the process starting in $y$ by $\bar{u}(z) \equiv u$ for $z<x$ and $\bar{u}(z)=\tilde{u}$ for $z\geq x$, where the constant control $u$ is chosen such that $c(u)>0$. %Denote the first hitting time of $x$ from $y$ by $\theta_x:=\inf\{t\geq0:X_t^u\geq x | X_0^u=y\}$ and define the control strategy $\bar{u}=(\bar{u}_t)$ for the process starting in $y$ by $\bar{u}_t\equiv u$ for $0\leq t\leq\theta_x$ and $\bar{u}_t=\tilde{u}_{t-\theta_x}$ for $t\geq\theta_x$, where the constant control $u$ is chosen such that $c(u)>0$. \\
We have
\begin{small}
\begin{align*}
 V(y)&\leq \Phi^{\bar{u}}(y)\leq\Pro\left(T_1 >\frac{x-y}{c(u)}\right)e^{-\delta\frac{x-y}{c(u)}}\Phi^{\tilde{u}}(x)+\Pro\left(T_1<\frac{x-y}{c(u)}\right) M\\
 &\leq e^{-(\delta+\lambda)\frac{x-y}{c(u)}}(V(x)+\varepsilon)+\left(1-e^{-\lambda\frac{x-y}{c(u)}}\right)M,
\end{align*}
\end{small}
which yields
\begin{footnotesize}
\begin{align*}
 |V(x)-V(y)|&=V(y)-V(x)\\
 &\leq V(x)\left(e^{-(\delta+\lambda)\frac{x-y}{c(u)}}-1\right)+\varepsilon e^{-(\delta+\lambda)\frac{x-y}{c(u)}}+\left(1-e^{-\lambda\frac{x-y}{c(u)}}\right)M
\end{align*}
\end{footnotesize}
Note that Lipschitz continuity implies absolute continuity of $V$.
\end{proof}
\section{Main Results}
Since we want to use the theory of stoachstic optimal control, it is crucial to show that the value function is a solution to the problem's Hamilton-Jacobi-Bellman equation (HJB). The proof follows similar arguments as the one of Lemma 3 in \cite{cani2017optimal}.
\begin{lemma}
 The value function $V(x)$ is on $(0,\infty)$ a.e. a solution to
 %\begin{small}
 \begin{align}\label{HJB}
 \begin{split}
  0=\inf_{u\in U}&\left\{c(u)V'(x) -(\delta+\lambda)V(x)+\lambda\int_0^{\rho(x,u)}V(x-r(y,u))\,dF_Y(y)\right.\\
  &+\lambda\left.\int_{\rho(x,u)}^\infty w(x,r(y,u)-x)\,dF_Y(y)\right\}.
  \end{split}
 \end{align}
%\end{small}
\end{lemma}
Here, $\rho(\cdot,u)=r(\cdot,u)^{-1}$ denotes the inverse of the retention function in the first component.
\begin{proof}
 We first show the $\leq$ part. Note that by continuity of $V$, the \textit{dynamic programming principle} holds, that is
 \begin{equation}\label{dynamic_pp}
  V(x)=\inf_{u\in\mathcal{U}} \E_x\left[e^{-\delta S}V(X_{S}^u)\mathds{1}_{S<\tau_x^u}+e^{-\delta \tau_x^u} w(X_{\tau_x^u-}^u,|X_{\tau_x^u}^u|)\mathds{1}_{S\geq\tau_x^u}\right],
 \end{equation}
 for every stopping time $S$. Next fix $x>0$, $h>0$ and $u\in U$ such that $c(u)>0$. % With $x_1:=x_0+hc(u)$,
 Consider the strategy $\hat{u}_t\equiv u$ for $t\in[0,h]$ and $\hat{u}_t=\tilde{u}_{t-h}$ for $t>h$ for some $\tilde{u}\in\mathcal{U}$. With $T_1$ again being the time of the first claim, set $S:=\min\{h,T_1\}$. Obviously, $S$ is a stopping time and the strategy $\hat{u}$ is constant in the time interval $[0,S]$. Setting $V(x)=0$ for $x<0$ and using \eqref{dynamic_pp}, we have
 \begin{align*}% now the strategy depends on time... change here or above?
  0\leq \E_{x}\left[e^{-\delta S} V(X_s^{\hat{u}})\right]-V(x)+\E_{x}\left[e^{-\delta \tau_{x}^{\hat{u}}}w(X_{\tau_{x}^{\hat{u}}-}^{\hat{u}},|X_{\tau_{x}^{\hat{u}}}^{\hat{u}}|)\mathds{1}_{S\geq\tau_{x}^{\hat{u}}}\right].
 \end{align*}
 Applying Dynkin's formula yields
 \begin{align*}
  0\leq \quad &\E_{x}\left[V(x)+\int_0^Se^{-\delta t}\left(\mathcal{A}^{\hat{u}}V(X_t^{\hat{u}})-\delta V(X_t^{\hat{u}})\right) dt\right]-V(x)\\
  &+\E_{x}\left[e^{-\delta \tau_{x}^{\hat{u}}}w(X_{\tau_{x}^{\hat{u}}-},|X_{\tau_{x}^{\hat{u}}}|)\mathds{1}_{S\geq\tau_{x}^{\hat{u}}}\right],
 \end{align*}
 where $\mathcal{A}^{\hat{u}}$ denotes the generator of the process $X_t^{\hat{u}}$, which, according to \cite{rolski2009stochastic}, Theorem 11.2.2, is given by
 \begin{equation}\label{generator}
  \A^{\hat{u}}g(x)=c(\hat{u}_t)g'(x)-\lambda g(x)+\lambda \int_0^\infty g(x-r(y,\hat{u}_t))\,dF_Y(y).
 \end{equation}
This leads to
\begin{align*}
 0\leq \E_{x}\left[\int_0^S e^{-\delta t}\Big( c(\hat{u}_t)V'(X_t^{\hat{u}})-(\delta+\lambda)V(X_t^{\hat{u}})\Big.\right.\\
 +\lambda\left.\int_0^{\rho(X_t^{\hat{u}},\hat{u}_t)}V(X_t^{\hat{u}}-r(y,\hat{u}_t)\,dF_Y(y)\Big.\Big)\right]\\
 +\E_{x}\left[e^{-\delta \tau_{x}^{\hat{u}}}w(X_{\tau_{x}^{\hat{u}}-},|X_{\tau_{x}^{\hat{u}}}|)\mathds{1}_{S\geq\tau_{x}^{\hat{u}}}\right].
 \end{align*}
Collecting the terms, dividing by $h$ and using that $\hat{u}=u$ for $t\in[0,S]$ gives
\begin{small}
\begin{align*}
 0\leq\,&\frac1h\,\E_{x}\left[\int_0^S e^{-\delta t}c(u)V'(x+c(u)t)\,dt\right]\\
 &+\frac1h\E_{x}\left[e^{-\delta T_1}w(x+c(u)T_1,|X_{T_1}^u|) \mathds{1}_{S\geq\tau_{x}^{\hat{u}}}\right]\\
 &+\frac1h\E_{x}\left[\int_0^S e^{-\delta t}\Bigg(-(\delta+\lambda)V(x+c(u)t)\Bigg.\right.\\
 &\left.\left.+\lambda\int_0^{\rho(x+c(u)t,u)}V(x+c(u)t-r(y,u))\,dF_Y(y)\right)dt\right].
\end{align*}
\end{small}
Having created an analogous situation as in the proof of Lemma 3 in \cite{cani2017optimal}, we can use the same arguments to deduce
\begin{align*}
 \begin{split}
  0\leq\inf_{u\in\mathcal{U}}&\left\{c(u)V'(x) -(\delta+\lambda)V(x)+\lambda\int_0^{\rho(x,u)}V(x-r(y,u))\,dF_Y(y)\right.\\
  &+\lambda\left.\int_{\rho(x,u)}^\infty w(x,r(y,u)-x)\,dF_Y(y)\right\}
  \end{split}
\end{align*}
which is the first half of the proof.\\
$ $\\
For the other direction, we fix $x>0$ and choose $h>0$ such that $x+\pi h>0$, where $\pi<0$ is again the premium under full reinsurance. Let $u^1$ be an $h^2$-optimal strategy for \eqref{dynamic_pp} and take again $S:=\min\{T_1,h\}$. Starting, as above, with \eqref{dynamic_pp}, we get
\begin{align*}
 0>&\,\,\E_x\left[e^{-\delta S} V(X_S^{u^1})\mathds{1}_{S<\tau_x^{u^1}}-V(x)\right]-h^2-\varepsilon h\\
 &+\E_x\left[e^{-\delta S} w(X_{\tau_x^{u^1}-}^{u^1},|X_{\tau_x^{u^1}}^{u^1}|)\mathds{1}_{S\geq\tau_x^{u^1}}\right].
\end{align*}
Conditioning on the time and height of the first claim and using the exponential distribution of the inter-claim times, this can be written as
\begin{small}
\begin{align*}
 0&> e^{-(\delta+\lambda)h}V\left(\tilde{x}_h\right)\\
 &+\E_x\left[\int_0^h\lambda e^{-(\delta+\lambda)t}\int_0^{\rho(\tilde{x}_t,u^1_t)}V\left(\tilde{x}_t-r(y,u^1_t)\right)\,dF_Y(y)\,dt\right]\\
 &+\E_x\left[\int_0^h\lambda e^{-(\delta+\lambda)t}\int_{\rho(\tilde{x}_t,u^1_t)}^\infty w\left(\tilde{x}_t,r(y,u^1_t)-\tilde{x}_t\right)\,dF_Y(y) \,dt\right] \\
 &-V(x)-h^2-h\varepsilon .
\end{align*}
\end{small}
Note that, to improve readability, we used the notational shortcuts $\tilde{x}_t:=x+\int_0^tc(u^1_s)ds$ and $\tilde{x}_h:=x+\int_0^hc(u^1_s)ds$.\\% and $u^1_t:=u^1(X_t^{u^1})$ resp. $u^1_s:=u^1(X_s^{u^1})$.\\
$ $\\
At this point, we can again follow the proof of Lemma 3 in \cite{cani2017optimal} to deduce that
\begin{align*}
 0 > \,&c(u_0^1)V'(x)-(\delta+\lambda)V(x)+\lambda\int_0^{\rho(x,u_0^1)}V(x-r(y,u_0^1))\,dF_Y(y)\\
  &+\lambda\int_{\rho(x,u_0^1)}^\infty w(x,r(y,u_0^1)-x)\,dF_Y(y)-\varepsilon.
\end{align*}
And letting $\varepsilon\rightarrow 0$ completes the proof.% \textbf{BUT THIS PROOF EXCLUDES THE CASE $x=0$!!!}
\end{proof}
Having shown that the value function is a solution to the HJB equation \eqref{HJB}, we now need to show that it is the only one (at least with some given analytical properties).
\section{Uniqueness of the solution to the HJB equation and Verification Statement}
Note that ruin can either occur by a claim that is bigger than the current reserve (claim ruin) or by decreasing the reserve with a negative premium until the reserve becomes negative (smooth ruin). Under certain conditions it can actually be advantageous to deliberately induce smooth ruin and thus choose the penalty $e^{-\delta\tau_x^u}w(0,0)$. Later, we will see that the possibility of smooth ruin causes changes in the analytical framework of the model.\\
$ $\\
Write $\mathcal{C}^{+,b}[0,\infty)$ for the set of positive, continuous and bounded functions on $[0,\infty)$ and define the operator $\G$ on $\mathcal{C}^{+,b}[0,\infty)$ as
\begin{align*}
 \G f(x):=&\inf_{u\in\mathcal{U}}\Big\{\E_x\left[e^{-\delta T_1}f(X_{T_1}^u)\mathds{1}_{T_1<\tau_x^u}\right]\Big.\\
 &+ \E_x\left[e^{-\delta T_1}w(X_{T_1-}^u,|X_{T_1}^u|) \mathds{1}_{T_1=\tau_x^u} \right]\\
 &+\left.\E_x\left[e^{-\delta \tau_x^u}w(0,0)\mathds{1}_{T_1>\tau_x^u}\right]\right\}. % need \mathds{1}_{X_{\tau^u-}^u=0} ?
\end{align*}
\begin{lemma}\label{continuity_G}
  $\G f\in\mathcal{C}^{+,b}[0,\infty)$. Furthermore, $\G$ is a contraction on $\mathcal{C}^{+,b}[0,\infty)$.
\end{lemma}
\begin{proof}
%We fist show that indeed, $\G f\in\mathcal{C}^{+,b}[0,\infty)$.
Positivity and boundedness follow immediately, since $w$ is assumed to have these properties. Now let $f\in \mathcal{C}^{+,b}[0,\infty)$ and $x,y\in [0,\infty)$ with $x>y$. With the same argumentation as in Lemma \ref{monotonicity}, we get that $\G f$ is monotonously decreasing. Choose $u_x$ as an $\varepsilon$-optimal strategy in $\G f(x)$ and write $\G^{u_x} f(x)$ for the right hand side of $\G f(x)$, with the control strategy $u_x$\\
In the following, we consider the reserve process pathwise. Write $\prescript{}{z}X_t^{u}$ for the risk process at time $t$, started in $z$ and controlled by the strategy $u$. Let $u_0\in U$ be the parameter corresponding to no reinsurance and define $\xi:=\inf\{t:\prescript{}{y}X_t^{u_0}=\prescript{}{x}X_t^{u_x}\}$, so $\xi$ is the time when the process started in $y$ hits the path of the process started in $x$. Now set the strategy $(u_y)_t\equiv u_0$ for $t\in[0,\xi]$ and $(u_y)_t=(u_x)_t$ for $t>\xi$. We have
\begin{align*}
 |\G f(x)-\G f(y)|=\G f(y) -\G f(x) \leq \G^{u_y} f(y) -\G^{u_x} f(x) +\varepsilon.
\end{align*}
Obviously, denoting the time of ruin of the process started in $x$ and controlled by the strategy $u_x$ by $\prescript{}{x}\tau^{u_x}$, we have $\prescript{}{x}\tau^{u_x}\geq\prescript{}{y}\tau^{u_y}$. Expanding the above equation gives %evtl nächste Gleichung rausnehmen
\begin{align*}
 \G^{u_y} f(y) -\G^{u_x} f(x) +\varepsilon = &\E\left[e^{-\delta T_1}f(\prescript{}{y}X_{T_1}^{u_y})\mathds{1}_{T_1<\prescript{}{y}\tau^{u_y}}\right]\\
 &+ \E\left[e^{-\delta T_1}w(\prescript{}{y}X_{T_1-}^{u_y},|\prescript{}{y}X_{T_1}^{u_y}|) \mathds{1}_{T_1=\prescript{}{y}\tau^{u_y}} \right]\\
 &+\E\left[e^{-\delta \tau^{u_y}}w(0,0)\mathds{1}_{T_1>\prescript{}{y}\tau^{u_y}}\right]\\%\mathds{1}_{\prescript{}{y}X_{\tau^{u_y}-}^{u_y}=0}
 &-\E\left[e^{-\delta T_1}f(\prescript{}{x}X_{T_1}^{u_x})\mathds{1}_{T_1<\prescript{}{x}\tau^{u_x}}\right]\\
 &- \E\left[e^{-\delta T_1}w(\prescript{}{x}X_{T_1-}^{u_x},|\prescript{}{x}X_{T_1}^{u_x}|) \mathds{1}_{T_1=\prescript{}{x}\tau^{u_x}} \right]\\
 &-\E\left[e^{-\delta \tau^{u_x}}w(0,0)\mathds{1}_{T_1>\prescript{}{x}\tau^{u_x}}\right] +\varepsilon. %\mathds{1}_{\prescript{}{x}X_{\tau^{u_x}-}^{u_x}=0}
\end{align*}
After collecting terms, we see that
\begin{align*}
 \G^{u_y} f(y) -\G^{u_x} f(x)=&\,\,\E\left[e^{-\delta T_1}(f(\prescript{}{y}X_{T_1}^{u_y})-f(\prescript{}{x}X_{T_1}^{u_x}))\mathds{1}_{T_1<\prescript{}{y}\tau^{u_y}}\right]\\
 &-\E\left[e^{-\delta T_1}f(\prescript{}{x}X_{T_1}^{u_x})\mathds{1}_{\prescript{}{y}\tau^{u_y}=T_1<\prescript{}{x}\tau^{u_x}}\right]\\
 &+ \E\left[e^{-\delta T_1}w(\prescript{}{y}X_{T_1-}^{u_y},|\prescript{}{y}X_{T_1}^{u_y}|) \mathds{1}_{\prescript{}{y}\tau^{u_y}=T_1=\prescript{}{x}\tau^{u_x}} \right]\\
 &-\E\left[e^{-\delta T_1}w(\prescript{}{x}X_{T_1-}^{u_x},|\prescript{}{x}X_{T_1}^{u_x}|) \mathds{1}_{\prescript{}{y}\tau^{u_y}=T_1=\prescript{}{x}\tau^{u_x}} \right].
\end{align*}
Note that the terms for smooth ruin before $T_1$ cancel out, since in this setting smooth ruin is only possible, after the processes started in $x$ and $y$ have merged. At this point it is helpful to distinguish the cases $\xi\leq T_1$ and $\xi>T_1$, so whether or not the merge has already happened before the first claim. Considering the summands separately yields
\begin{align*}
&\E\left[e^{-\delta T_1}(f(\prescript{}{y}X_{T_1}^{u_y})-f(\prescript{}{x}X_{T_1}^{u_x}))\mathds{1}_{T_1<\prescript{}{y}\tau^{u_y}}\right]\\
&=\E\left[e^{-\delta T_1}(f(\prescript{}{y}X_{T_1}^{u_y})-f(\prescript{}{x}X_{T_1}^{u_x}))\mathds{1}_{T_1<\prescript{}{y}\tau^{u_y}}\mathds{1}_{\xi>T_1}\right]\\
&+\underbrace{\E\left[e^{-\delta T_1}(f(\prescript{}{y}X_{T_1}^{u_y})-f(\prescript{}{x}X_{T_1}^{u_x}))\mathds{1}_{T_1<\prescript{}{y}\tau^{u_y}}\mathds{1}_{\xi\leq T_1}\right]}_{=0}.
\end{align*}
We see that for $\xi\leq T_1$ the terms cancel out. To analyze what happens for $\xi>T_1$, take $\varepsilon_c>0$ and define
\[
t^*=\inf\{t:\exists\varepsilon_t>0:c(u_x(\tilde{t}))<c-\varepsilon_c\text{ for }\tilde{t}\in[t,t+\varepsilon_t]\}.
\]
In other words, at $t^*$ starts the first open interval where the drift of the process started in $x$ is by at least $\varepsilon_c$ smaller than the drift of the process started in $y$. For $|x-y|$ small enough, this interval in time will, even for arbitrarily small $\varepsilon_c$, be enough for $\prescript{}{y}X_{t}^{u_y}$ to reach the trajectory of $\prescript{}{x}X_{t}^{u_x}$ so we know $\xi\in[t^*,t^*+\varepsilon_t]$ with $\varepsilon_t\to0$ for $|x-y|\to0$. Now let us consider the first claim occurence $T_1$.
\begin{itemize}
\item{For $T_1<t^*$, the processes haven't merged yet, but their premium rates are at most $\varepsilon_c$ apart and since the premium is a continuous, strictly monotone function, their control strategies are at most $\delta_c$ apart. Since the retention function is also continuous in $u$, and $\varepsilon_c$ was arbitrary, we know that
$|f(\prescript{}{y}X_{T_1}^{u_y})-f(\prescript{}{x}X_{T_1}^{u_x})|\to0$ as $|x-y|\to0$.}
\item{If $t^*<T_1<\xi$, we cannot directly control the difference in the jump at $T_1$, but since we know that $\xi\in[t^*,t^*+\varepsilon_t]$ and because the distribution of $T_1$ is continuous, $\Pro(T_1\in[t^*,t^*+\varepsilon_t])$ goes to zero for $\varepsilon_t\to0$.}% and $\varepsilon_t$ was chosen arbitrarily small.}
\end{itemize}
Similarly, for the second summand, we see that
\begin{align*}
&\E\left[e^{-\delta T_1}f(\prescript{}{x}X_{T_1}^{u_x})\mathds{1}_{\prescript{}{y}\tau^{u_y}=T_1<\prescript{}{x}\tau^{u_x}}\right]\\
&=\E\left[e^{-\delta T_1}f(\prescript{}{x}X_{T_1}^{u_x})\mathds{1}_{\prescript{}{y}\tau^{u_y}=T_1<\prescript{}{x}\tau^{u_x}}\mathds{1}_{\xi<T_1}\right].
\end{align*}
Using the definition of $t^*$ as before, we have again two cases to consider.
\begin{itemize}
\item{For $T_1<t^*$ we already argued that the two paths of the process are arbitrarily close for $|x-y|$ being sufficiently small. Since for a claim that ruins the process started in $y$ but not the one started in $x$, we know that the claim height $Y_1$ must be in $[\prescript{}{y}X_{T_1-}^{u_y},\prescript{}{x}X_{T_1-}^{u_x}]$ and since the claim height distribution is assumed to be continuous, we deduce $\Pro(Y_1\in[\prescript{}{y}X_{T_1-}^{u_y},\prescript{}{x}X_{T_1-}^{u_x}])\to0$ for $|x-y|\to0$.}
\item{In the case $t^*<T_1<\xi$, we can use the same argumentation as above to reach the conclusion that $\Pro(T_1\in[t^*,t^*+\varepsilon_t])$ goes to zero for $|x-y|\to0$.}
\end{itemize}
A combination of the arguments we used so far and exploiting the continuity of $w$ will also send the remaining two summands to $0$, showing continuity of $\G f$.\\
$ $\\
It remains to show that $\G$ is a contraction on $\mathcal{C}^{+,b}[0,\infty)$. so let $f_1$, $f_2$ be positive, continuous and bounded and $u_1$,$u_2$ be their minimizing strategies in $\G$. We have
 \begin{footnotesize}
 \begin{align*}
  \G &f_1(x)-\G f_2(x)\\
  &=\inf_{u\in\U}\left\{\E_x\left[e^{-\delta T_1}f_1(X_{T_1}^u)\mathds{1}_{T_1<\tau^u}\right]+E_x\left[e^{-\delta T_1}w(X_{T_1-}^u,|X_{T_1}^u|)\mathds{1}_{T_1=\tau^u} \right]\right.\\
  &+\left.\E_x\left[e^{-\delta \tau^u}w(0,0)\mathds{1}_{T_1>\tau^u}\right]\right\}\\ %\mathds{1}_{X_{\tau^u-}^u=0}
  &-\inf_{u\in\U}\left\{\E_x\left[e^{-\delta T_1}f_2(X_{T_1}^u)\mathds{1}_{T_1<\tau^u}\right]+E_x\left[e^{-\delta T_1}w(X_{T_1-}^u,|X_{T_1}^u|)\mathds{1}_{T_1=\tau^u} \right]\right.\\
  &-\left.\E_x\left[e^{-\delta \tau^u}w(0,0)\mathds{1}_{T_1>\tau^u}\right]\right\}\\ %\mathds{1}_{X_{\tau^u-}^u=0}
  &\leq \E_x\left[e^{-\delta T_1}(f_1(X_{ T_1}^{u_2})-f_2(X_{ T_1}^{u_2}))\mathds{1}_{T_1<\tau^{u_2}}\right]\\
  &=\int_0^\infty \! e^{-\delta t} \lambda e^{-\lambda t} \int\limits_0^{\rho(X_{t-}^{u_2},u_2)} \!f_1(X_{t-}^{u_2}-r(y,u_2))-f_2(X_{t-}^{u_2}-r(y,u_2))dF_Y(y)dt\\
  &\leq \underbrace{\E\left[e^{-\delta T_1}\right]}_{< 1} ||f_1-f_2||_\infty.
 \end{align*}
 \end{footnotesize}
\end{proof}
From the definition of $\G$, we see that $\G V=V$ holds by the dynamic programming principle. In the following, we want to establish the connection between $\G$ and the HJB equation.
\begin{lemma}\label{fixed_point}
Let $f\in\mathcal{C}^{+,b}[0,\infty)$ be a solution to the HJB equation \eqref{HJB} with $f(0)\leq w(0,0)$. For $x\in(0,\infty)$ set 
\begin{align*}
u_f(x)=\argmin_{u\in U}&\left\{c(u)f'(x)-(\delta+\lambda)f(x)+\lambda\int_0^{\rho(x,u)}f(x-r(y,u))\,dF_Y(y)\right.\\
  &+\lambda\left.\int_{\rho(x,u)}^\infty w(x,x-r(y,u))\,dF_Y(y)\right\}.
\end{align*}
We complement the definition of $u_f$ by taking
\[
 u_f(0)=\begin{cases}
         u_f(0+)& \text{ if } f(0)< w(0,0)\\
         u^*&\text{ if } f(0)=w(0,0),
        \end{cases}
\]
where $u^*$ denotes the strategy of full reinsurance. Then $f$ is a fixed point of $\G$ and $u_f$ is the minimizing strategy.
%  Any $f\in\mathcal{C}^{+,b}[0,\infty)$ that is a solution to the HJB equation \eqref{HJB} and in addition fulfills $f(0)=w(0,0)$ if $c(u_f(0+))<0$ is a fixed point of $\G$. Here, $u_f$ is the minimizing strategy for $f$ in \eqref{HJB}.
\end{lemma}
\begin{remark}
  In the above Lemma, we write $u_f(x)$ to indicate that we are working with a Markov control, i.e. solely dependent on the current state. % By construction, for every $x>0$, we take $u_f(x)$ as the minimizer in \eqref{HJB} and for $x=0$ we can set $u_f(0)=u_f(0+)$. The usage of the right hand limit here is justified, since similar arguments as in Lemma 2.14 of \cite{schmidli2007stochastic}, but using Lipschitz continuity instead of continuity of $f'$, yield that the sign of $c(u_f(x))$ does not change in the limit $x\downarrow 0$. Consequently, a negative drift leads to immediate ruin for $x=0$ and one could also identify $u_f$ with full reinsurance in this situation. Otherwise, a positive premium rate takes the process immediately to the domain of the HJB equation.
%Therefore, the control process associated to the function $f$, given by $u_t=u_f(X^{u_f}_{t-})$ is a well specified Markov control and is admissible.
Furthermore, the described choice of $u_f(x)$ happens in a measurable way, as can be seen from arguments similar to those of Lemma 2.12 in \cite{schmidli2007stochastic}.
\end{remark}
% \begin{remark}The HJB equation only describes the behavior on the interior of $[0,\infty)$. However, $u_f(0)<0$ means
% that the process can make the transition from ``alive'' to ``ruined'' without a jump. Interpreting the process $X_t^u$ as a piecewise deterministic Markov process, this means the active boundary $\Gamma$ is not empty here, which in turn requires to incorporate the additional boundary condition $f(0)=w(0,0)$. Because smooth ruin is usually not considered in reinsurance scenarios where the ruin probability or dividend payments are to be optimized, it is an interesting feature of our model to (potentially) have $\Gamma\neq\emptyset$. For more details on this subject, we refer to Chapter $11$ of \cite{rolski2009stochastic}.
% \end{remark}
\begin{remark}
The aim of this section is to show that the function $f$ with the properties of Lemma \ref{fixed_point} actually is the value function $V$. So demanding that $f(0)\leq w(0,0)$ is a natural condition since it is certainly fulfilled by $V$. The definition of $u_f$ is also very intuitive as can be seen by the following consideration. Having $u_f(0)=u^*$ means a negative premium in zero and therefore the process can make the transition from ``alive'' to ``ruined'' without a jump. Interpreting the process $X_t^{u_f}$ as a piecewise deterministic Markov process (PDMP), this means the active boundary $\Gamma$ is not empty here, which, in the theory of PDMPs, goes along with the additional boundary condition $f(0)=w(0,0)$. Because smooth ruin is usually not considered in reinsurance scenarios where the ruin probability or dividend payments are to be optimized, it is an interesting feature of our model to (potentially) have $\Gamma\neq\emptyset$. For more details on this subject, we refer to Chapter $11$ of \cite{rolski2009stochastic}.
\end{remark}
\begin{proof}
 We start with the HJB equation
\begin{align*}
 0=\inf_{u\in U}&\left\{c(u)f'(x) -(\delta+\lambda)f(x)+\lambda\int_0^{\rho(x,u)}f(x-r(y,u))\,dF_Y(y)\right.\\
  +&\lambda\left.\int_{\rho(x,u)}^\infty w(x,x-r(y,u))\,dF_Y(y)\right\}.
\end{align*}
This holds for arbitrary $x$ and $f$ is certainly defined at all $X_t^{u}$ for $t\in[0,T_1\wedge\tau^{u}_x]$. Denoting the minimizing strategy by $u_f$ (which exists by the continuity of all involved functions) and using Dynkin's formula, we can write
\begin{small}
\begin{align*}
  0=\,\,&\E_x\left[\int_0^{T_1\wedge\tau^{u_f}}e^{-\delta t}\Bigg(c(u_f)f'(X_{t-}^{u_f})-(\delta+\lambda)f(X_{t-}^{u_f})\Bigg.\right.\\
  &+\lambda\int_0^{\rho(X_{t-}^{u_f},u_f)}f(X_{t-}^{u_f}-r(y,u_f))\,dF_Y(y)\\
  &\left.\Bigg.+\lambda\int_{\rho(X_{t-}^{u_f},u_f)}^\infty w(X_{t-}^{u_f},r(y,u_f)-X_{t-}^{u_f})\,dF_Y(y)\Bigg)dt\right]\\
  &=\E_x\left[e^{-\delta( T_1\wedge\tau^{u_f})} f(X_{T_1\wedge\tau^{u_f}}^{u_f})\mathds{1}_{T_1\neq\tau^{u_f}}\right]-f(x)\\
  %&+\E_x\left[\int_0^{T_1\wedge\tau^{u_f}}e^{-\delta t}\lambda\int_{\rho(X_{t-}^{u_f},u_f)}^\infty w(X_{t-}^{u_f},r(y,u_f)-X_{t-}^{u_f})\,dF_Y(y)dt\right]\\
  &+\E_x\left[\int_0^{T_1}e^{-\delta t}\lambda\int_{\rho(X_{t-}^{u_f},u_f)}^\infty w(X_{t-}^{u_f},r(y,u_f)-X_{t-}^{u_f})\,dF_Y(y)dt\,\mathds{1}_{T_1=\tau^{u_f}}\right]\\
  &=\E_x\left[e^{-\delta T_1} f(X_{T_1}^{u_f})\mathds{1}_{T_1<\tau^{u_f}}\right]+\E_x\left[e^{-\delta \tau^{u_f}} f(0)\mathds{1}_{T_1>\tau^{u_f}}\right]-f(x)\\
  %&+\E_x\left[\int_0^{T_1\wedge\tau^{u_f}}e^{-\delta t}\lambda\int_{\rho(X_{t-}^{u_f},u_f)}^\infty w(X_{t-}^{u_f},r(y,u_f)-X_{t-}^{u_f})\,dF_Y(y)dt\right]\\
  &+\E_x\left[\int_0^{T_1}e^{-\delta t}\lambda\int_{\rho(X_{t-}^{u_f},u_f)}^\infty w(X_{t-}^{u_f},r(y,u_f)-X_{t-}^{u_f})\,dF_Y(y)dt\,\mathds{1}_{T_1=\tau^{u_f}}\right].
\end{align*}
\end{small}
% Rewriting the last term gives% and using the compensation theorem yields
% \begin{small}
% \begin{align*}
%  &\E_x\left[\int_0^{T_1\wedge\tau^{u_f}}e^{-\delta t}\lambda\int_{\rho(X_t^{u_f},u_f)}^\infty w(X_t^{u_f},r(y,u_f)-X_t^{u_f})\,dF_Y(y)dt\right]\\
%  &=\E_x\left[\int_0^{T_1}e^{-\delta t}\lambda\int_{\rho(X_{t-}^{u_f},u_f)}^\infty w(X_{t-}^{u_f},r(y,u_f)-X_{t-}^{u_f})\,dF_Y(y)dt\mathds{1}_{T_1\leq\tau^{u_f}}\right]\\
%  &+\E_x\left[\int_0^{\tau^{u_f}}\!e^{-\delta t}\lambda\int\limits_{\rho(X_{t-}^{u_f},u_f)}^\infty w(X_{t-}^{u_f},r(y,u_f)-X_{t-}^{u_f})\,dF_Y(y)dt\mathds{1}_{T_1>\tau^{u_f}}\right].
%  %&=\E_x\left[e^{-\delta T_1}w(X_{T_1-},|X_{T_1}|)\mathds{1}_{T_1=\tau^{u_f}}\right]+e^{-(\delta+\lambda)\bar{\tau}^{u_f}}w(0,0).
% \end{align*}
% \end{small}
% Note that left hand limits $t-$ are justified on $[0,T_1)$.
We now use the compensation theorem
\[
  \E_x\left[\int_0^{T_1}\lambda e^{-\delta t}H_t\,dt\right]=\E_x\left[\int_0^{T_1}e^{-\delta t} H_t\,dN_t\right]=\E_x\left[e^{-\delta T_1}H_{T_1}\right],
 \]
 where $\lambda$ is the intensity of the counting process $N_t$, for the previsible process
 \[
 H_t:=\int_{\rho(X_{t-}^{u_f},u_f)}^\infty w(X_{t-}^{u_f},r(y,u_f)-X_{t-}^{u_f})\,dF_Y(y).
 \]
 Taking $u_f(0)$ as in the statement of the lemma yields
 \begin{align*}
  f(x)=&\,\,\E_x\left[e^{-\delta T_1}f(X_{T_1}^{u_f})\mathds{1}_{T_1<\tau^{u_f}}\right]+\E_x\left[e^{-\delta T_1}w(X_{T_1-}^{u_f},|X_{T_1}^{u_f}|)\mathds{1}_{T_1=\tau^{u_f}}\right]\\
  &+\E_x\left[e^{-\delta \tau^u}w(0,0)\mathds{1}_{T_1>\tau^{u_f}}\right], %\mathds{1}_{X_{\tau^{u_f}-}^{u_f}=0}
 \end{align*}
because $\mathds{1}_{T_1>\tau^{u_f}}=0$ if $c(u_f(0))\geq0$. So we showed $f\geq \mathcal{G}f$.\\
$ $\\
On the other hand
 \begin{footnotesize}
 \begin{align*}
  \G f(x)=&\inf_{u\in\U}\left\{\E_x\left[e^{-\delta T_1}f(X_{T_1}^u)\mathds{1}_{T_1<\tau^u_x}\right]+\E_x\left[e^{-\delta T_1}w(X_{T_1-}^u,|X_{T_1}^u|)\mathds{1}_{T_1=\tau^u_x} \right]\right.\\
  &+\left.\E_x\left[e^{-\delta \tau^u_x}w(0,0)\mathds{1}_{T_1>\tau^u_x}\right]\right\}\\ %\mathds{1}_{X_{\tau^u-}^u=0}
 %  &=\inf_{u\in\tU}\left\{f(x)+\E_x\left[\int_0^{T_1}e^{-\delta t}\left(c(u)f'(X_t^{u})-(\delta+\lambda)f(X_t^{u})+\lambda\int_0^{\rho(X_t^{u},u_f)}f(X_t^{u}-r(y,u))\,dF_Y(y)\right)\right.\right. \\
 %   &+\Bigg.\E_x\left[e^{-\delta T_1}w(X_{T_1-}^u,|X_{T_1}^u|)\mathds{1}_{T_1=\tau}\right]\Bigg\}\\
 %  &=\inf_{u\in\tU}\left\{\E_x\left[f(x)-e^{-\delta T_1}w(X_{T_1-}^{u_f},|X_{T_1}^{u_f}|)\mathds{1}_{T_1=\tau^{u_f}}+e^{-\delta T_1}w(X_{T_1-}^{u_f},|X_{T_1}^{u_f}|)\mathds{1}_{T_1=\tau^{u}}\right]\right\}
 &=\inf_{u\in\U}\left\{f(x)+\E_x\left[\int_0^{T_1\wedge\tau^u}e^{-\delta t}\Bigg(c(u)f'(X_{t-}^{u})-(\delta+\lambda)f(X_{t-}^{u})\Bigg.\right.\right.\\
 &+\left.\left.\lambda\int_0^{\rho(X_{t-}^{u},u)}f(X_{t-}^{u}-r(y,u))\,dF_Y(y)\right)dt\right] \\
 &+\E_x\left[e^{-\delta \tau^u_x}w(0,0)\mathds{1}_{T_1>\tau^u_x}\right]\\
 &+\Bigg.\E_x\left[\int_0^{T_1\wedge\tau^u}e^{-\delta t} \lambda\int_{\rho(X_{t-}^u,u)}^\infty w(X_{t-}^u,r(y,u)-X_{t-}^u)\,dF_Y(y)\,dt\right]\Bigg\}\\
 &\geq f(x)
  \end{align*}
 \end{footnotesize}
 where we again used the compensation theorem for the last expression and the last inequality follows from the HJB equation.\\
\end{proof}
The following Theorem is an immediate consequence of Lemmas \ref{continuity_G} and \ref{fixed_point} combined with Banach's fixed point theorem. It is also the central statement of this section as it establishes the HJB euqation as the crucial tool for finding the value function.
 \begin{theorem}
  In the function space $\mathcal{C}^{+,b}[0,\infty)$, the value function $V$ is the unique fixed point of $\G$ and hence it is also the unique solution to the HJB equation.
 \end{theorem}

\section{Numerical Examples}
Following the results in the previous section, we can construct the value function by finding a solution to the Hamilton-Jacobi-Bellman equation. Our method of choice was the \textit{policy iteration} (for a detailled review of applicable methods see e.g. \cite{kushner2013numerical}). In a first step, we discretized the interval $[x_0,x_N]$ where we want to find the solution. Then we started with the generic strategy $u_0$ of no reinsurance and used Monte-Carlo techniques to find the values for $\Phi^{u_0}(x_0)$ and $\Phi^{u_0}(x_N)$. Knowledge of these boundary values then enabled us to numerically solve the integro-differential equation that is given by the Feynman-Kac type equation
\begin{align*}
 0=&\,c(u_0)(\Phi^{u_0})'(x) -(\delta+\lambda)V(x)+\lambda\int_0^{\rho(x,u_0)}\Phi^{u_0}(x-r(y,u_0))\,dF_Y(y)\\
  &+\lambda\int_{\rho(x,u_0)}^\infty w(x,r(y,u_0)-x)\,dF_Y(y)
\end{align*}
as it is derived in Theorem 11.2.3 of \cite{rolski2009stochastic}. Here, we used a finite differences approach. Having calculated $\Phi^{u_0}(x)$ for all $x$ on the grid corresponding to $[x_0,x_N]$ in this manner, we look for an improving strategy by taking
\begin{small}
 \begin{align*}
  u^{(1)}(x)=&\argmin_{u\in\U}\Bigg\{ c(u)(\Phi^{u_0})'(x) -(\delta+\lambda)V(x)\Big.\\
  &+\lambda\int_0^{\rho(x,u)}\Phi^{u_0}(x-r(y,u))\,dF_Y(y)\\
  &+\left.\lambda\int_{\rho(x,u)}^\infty w(x,r(y,u)-x)\,dF_Y(y)\right\}.
 \end{align*}
\end{small}
Now we repeat the procedure with $u^{(1)}$ in place of $u_0$ to construct $u^{(2)},u^{(3)},\dots$ until no significant improvement can be achieved anymore.\\
$ $\\
For referencing, we chose similar parameters as in \cite{schmidli2007stochastic} Chapter 2 for the risk model. That is, we set the Poisson intensity $\lambda$ to $1$ and the interval under consideration to $[0,14]$. The reinsurance shall be of proportional type, i.e. the retention function is given as $r(y,u)=u\cdot y$ for $u\in[0,1]$. Furthermore, we calculated the premiums $c(u)$ following the expected value principle with the cedent's safety loading denoted by $\eta$ and the reinsurer's safety loading $\theta$. So
\[
 c(u)=\lambda\beta(\eta-\theta+u(1+\theta)),
\]
where $\beta$ denotes the expected claim height. In all examples, we set $\eta=0.5$ and $\theta=0.7$.

\subsection{Exponential Claims}
First, we want to consider exponentially distributed claims. Setting the expected claim height to $1$, this means $F_Y(y)=1-e^{-y}$. We start with the very simple penalty function $w_1(x,y)=1$, so we want to minimize the discounted ruin probability. This exact setting was treated in \cite{schmidli2007stochastic} for $\delta=0$. We undertook the calculation for the case $\delta=0.05$ to see the effect of the discount factor on value function and strategy. The resulting strategy and the first $5$ iterations of $\Phi$ are shown in Figures \ref{u_exp_05} and \ref{V_comp_exp_05} respectively. While Figure \ref{u_exp_05} shows clear resemblance to the undiscounted case in \cite{schmidli2007stochastic}, we see that in Figure \ref{V_comp_exp_05}, the difference between the first $3$ Gerber Shiu functions (blue, red, yellow) is still significant, whereas there is almost no difference anymore between functions $3,4$ and $5$ (depicted yellow, purple and green).\\% Thus we call the fith cost function the value function here.\\
\begin{figure}[htbp]
	% minipage mit (Blind-)Text
	\begin{minipage}{0.5\textwidth} 
	\includegraphics[width=\textwidth]{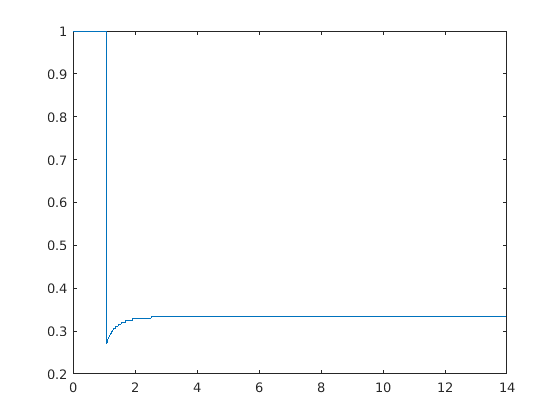}
	\caption{\footnotesize Optimal strategy for exponential claims}
	\label{u_exp_05}
	\end{minipage}
	% Auffüllen des Zwischenraums
	\hfill
	% minipage mit Grafik
	\begin{minipage}{0.5\textwidth}
	% \textwidth bezieht sich nun auf die Minipage
	\includegraphics[width=\textwidth]{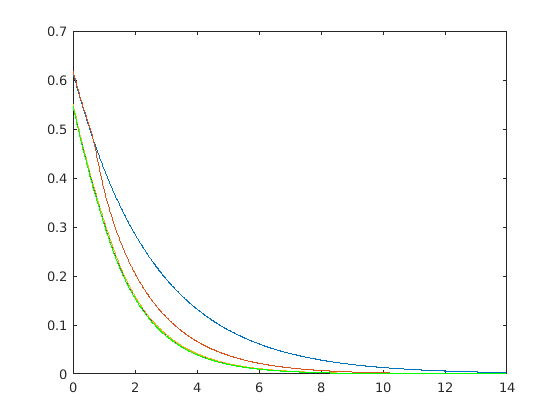}
	\caption{\footnotesize functions $\Phi^{u_1}$ to $\Phi^{u_5}$.\label{V_comp_exp_05}} 
	\end{minipage}
\end{figure}
$ $\\
To show the flexibility of our approach we want to consider a more general penalty function. So we will now use $w_2(x,y)=\min(10^{10},(x+0.5)(y+1)^2)$ and also increase the discounting rate to $\delta=0.1$. This choice of penalty function might seem arbitrary or hypothetical at first, but making the penalty actually depend on the suprlus prior to- and deficit at ruin will trigger the incentive for smooth ruin in some situations. As before, we used policy iteration and stopped when improvements fell under a predefined level. In Figure \ref{hjb_w2_exp}, we plotted the corresponding value of the HJB equation. In the optimum this value is zero, values close to zero indicate a good approximation. The optimal strategy can be seen in Figure \ref{u_w2_exp_08} where the red line is drawn at $0.1176$, the zero of the premium function $c(u)$. So for $u<0.1176$, the total premiums are negative.\\
\begin{figure}[htbp]
	% minipage mit (Blind-)Text
	\begin{minipage}{0.5\textwidth} 
	\includegraphics[width=\textwidth]{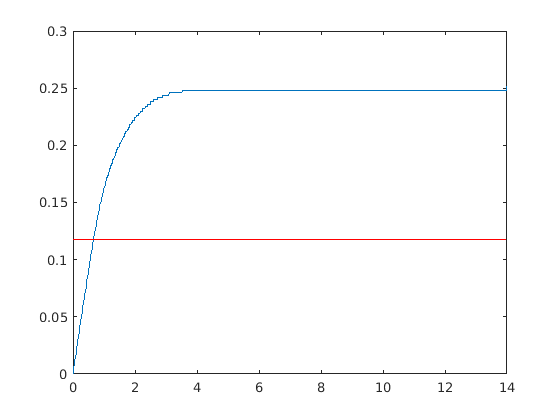}
	\caption{\footnotesize Optimal strategy after 7 iterations, sign change at red line.}
	\label{u_w2_exp_08}
	\end{minipage}
	% Auffüllen des Zwischenraums
	\hfill
	% minipage mit Grafik
	\begin{minipage}{0.5\textwidth}
	% \textwidth bezieht sich nun auf die Minipage
	\includegraphics[width=\textwidth]{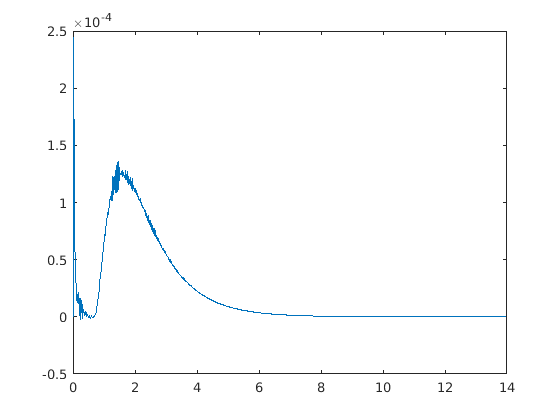}
	\caption{\footnotesize Value of HJB equation.\label{hjb_w2_exp}} 
	\end{minipage}
\end{figure}
\begin{figure}[htbp]
	\begin{center}
	\includegraphics[width=0.66\textwidth]{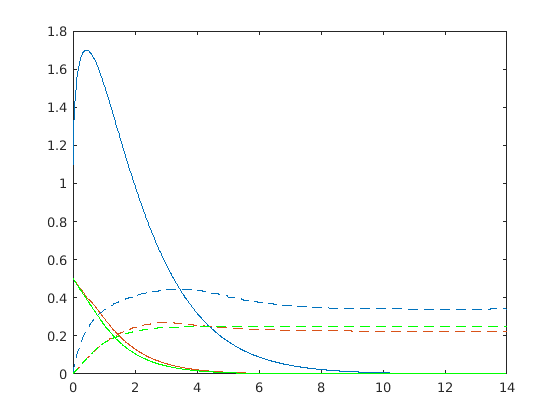}
	\caption{\footnotesize functions $\Phi^{u_2}$, $\Phi^{u_4}$ and $\Phi^{u_6}$ along with strategies $u_2$, $u_4$ and $u_6$.}
	\label{V_u_comb}
	\end{center}
\end{figure}
The resulting strategy is particularly interesting since it leads to smooth ruin. That means, for low reserve values, the insurer prefers deliberately terminating the business and paying the comparably low penalty $w_2(0,0)=0.5$ instead of taking the risk of a much higher penalty. In Figure \ref{V_u_comb}, we show the second (blue), fourth (red) and sixth (green) cost function with the respective minimizing strategies (dashed lines in the corresponding colors).
\subsection{Pareto Claims}
In insurance mathematics, a particular interest lies in the study of heavy-tailed distributions. To account for that, we also investigated the case of pareto distributed claims. For $w_1$, that is the discounted ruin probability, we chose the claim distribution $F_Y(x)=1-(x+1)^{-2}$, resulting again in an expected claim height of $1$. This claim distribution was also used in \cite{schmidli2007stochastic}. The resultig strategy is shown in Figures \ref{u_par_05}, while Figure \ref{V_comp_par_05} gives again the first $5$ cost functions in the order blue, red, yellow, purple and green.\\
\begin{figure}[htbp]
	% minipage mit (Blind-)Text
	\begin{minipage}{0.5\textwidth} 
	\includegraphics[width=\textwidth]{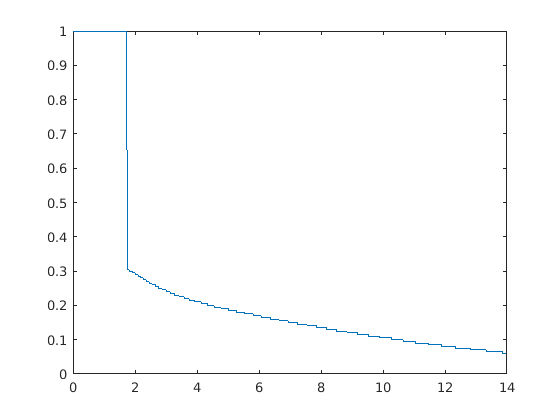}
	\caption{\footnotesize Optimal strategy for Pareto claims}
	\label{u_par_05}
	\end{minipage}
	% Auffüllen des Zwischenraums
	\hfill
	% minipage mit Grafik
	\begin{minipage}{0.5\textwidth}
	% \textwidth bezieht sich nun auf die Minipage
	\includegraphics[width=\textwidth]{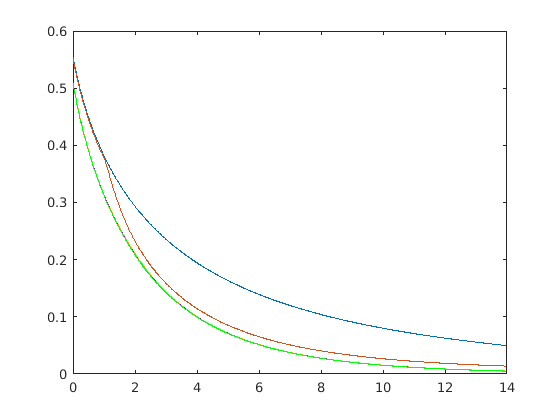}
	\caption{\footnotesize functions $\Phi^{u_1}$ to $\Phi^{u_5}$.\label{V_comp_par_05}}%\label{hjb_par}} 
	\end{minipage}
\end{figure}
% \begin{figure}[htbp]
% 	\begin{center}
% 	\includegraphics[width=0.7\textwidth]{V1_V5_par_delta05.png}
% 	\caption{\footnotesize cost functions.}
% 	\label{V_comp_par_05}
% 	\end{center}
% \end{figure}
%\newpage
$ $\\
As for the exponential case, we also want to find the optimal strategy for pareto distributed claims and the penalty function $w_2$. Since the second moment for pareto distributions exists only for shape parameters greater than $2$, we chose the claim height distribution $F_Y(x)=1-(1+x)^{-3}$. In Figure \ref{u_w2_par_05}, we again added the red line at the zero of $c(u)$. Note that on the whole interval the optimal strategy leads to negative premiums. This can be explained by the heavy tails of the Pareto distribution. At no level of the reserves does the chance to survive but under the risk of a potentially heavy ruin, outweigh the very moderate penalty of $w_2(0,0)e^{-\delta \tau}<0.5$. In Figures \ref{V_w2_comp_par} and \ref{w2_par_u_comp}, we also plotted the second to fith iteration of the value function resp. the corresponding strategy.
\begin{figure}[htbp]
	% minipage mit (Blind-)Text
	\begin{minipage}{0.5\textwidth} 
	\includegraphics[width=\textwidth]{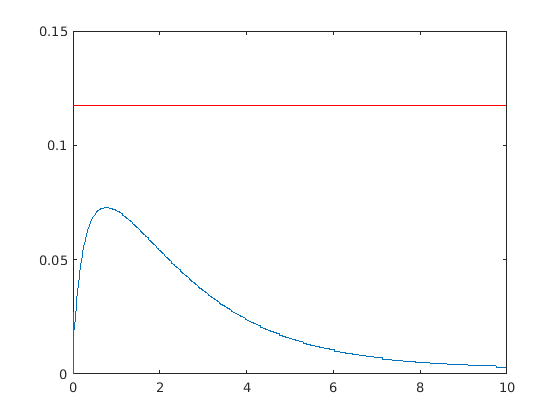}
	\caption{\footnotesize Optimal strategy for Pareto claims}
	\label{u_w2_par_05}
	\end{minipage}
	% Auffüllen des Zwischenraums
	\hfill
	% minipage mit Grafik
	\begin{minipage}{0.5\textwidth}
	% \textwidth bezieht sich nun auf die Minipage
	\includegraphics[width=\textwidth]{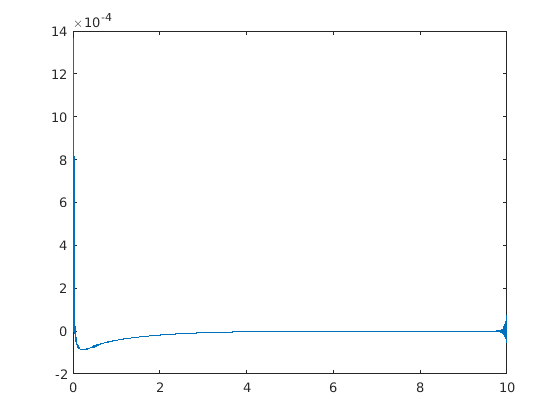}
	\caption{\footnotesize HJB error.\label{hjb_w2_par}} 
	\end{minipage}
	\begin{minipage}{0.5\textwidth}
	\includegraphics[width=\textwidth]{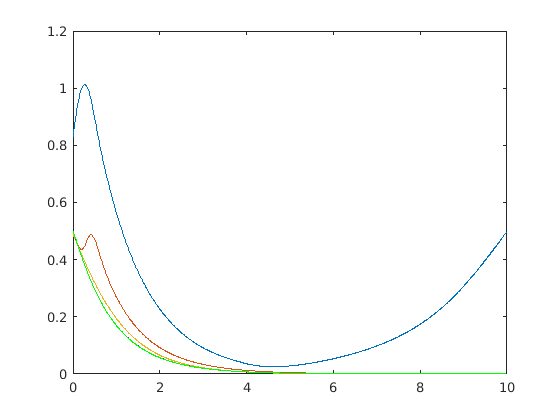}
	\caption{\footnotesize functions $\Phi^{u_2}$ to $\Phi^{u_5}$ .}
	\label{V_w2_comp_par}
	\end{minipage}
	\hfill
	\begin{minipage}{0.5\textwidth}
	 \includegraphics[width=\textwidth]{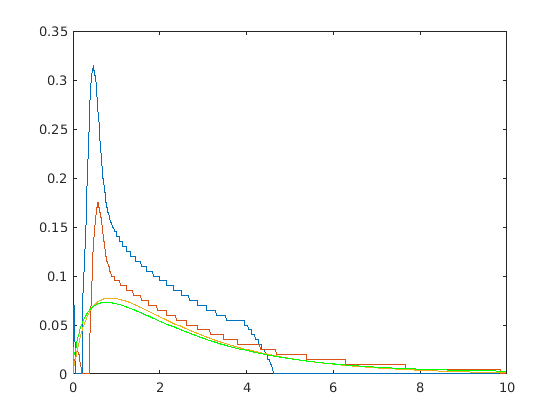}
	 \caption{\footnotesize strategies $u_2$ to $u_5$.\label{w2_par_u_comp}} 
	\end{minipage}
\end{figure}
\newpage
\subsection{A Note on the Numerics}
The calculations that were undertaken for this section turned out to be more laborious than expected. While some cases, like exponential claims without or with low discounting factor or pareto claims without discounting factor didn't make much trouble, other cases, namely the more general penalty function $w_2$ in combination with discount rates and Pareto claims were quite demanding. The reason for this is that the finite differences approach in these cases was extremely sensitive to the right starting value, indeed to an extent where MC techniques could not provide the needed accuracy anymore. Relying on IDE solvers that treat the problem in a more continuous way is not immediately possible, since strategies crossing the zero of $c(u)$ result in singularities in the involved ODE terms.\\
The method that brought the best results was an individually chosen mix of central and backwards differences combined with a MC simulation for an initioal guess, followed by a somewhat manual bisection technique to provide the correct initial values.\\
\subsection{Asymptotic Behaviour}
We also investigated the question of the asymptotically optimal strategy. In the case of exponentially distributed (that means light-tailed) claims, it is straightforward to proceed as in \cite{hald2004maximisation}. One has to keep in mind though that the presence of a discount factor $\delta$ changes the associated Lundberg equation to
\begin{equation}\label{lundberg}
 \lambda  \hat{m}_Y(\alpha)=1+\frac\delta\lambda+\frac{\alpha c}{\lambda} 
\end{equation}
where $\hat{m}_Y(\alpha)=\E\left[e^{\alpha Y}\right]$ is the moment-generating function of the claim height distribution $F_Y$. The positive solution $\gamma$ for which \eqref{lundberg} becomes zero (if such a solution exists) is usually called the adjustment coefficient. Now consider the Cram\'{e}r-Lundberg approximation for $\Psi(x)$, the ruin probability with initial value $x$, which reads
\begin{equation}\label{CLapprox}
 \lim_{x\to\infty} \Psi(x) e^{x\gamma} = C_\delta
\end{equation}
for some constant $C_\delta$. From \eqref{CLapprox}, it becomes clear that maximizing the adjustment coefficient by means of the reinsurance parameter will lead to the maximally fast asymptotical decay rate for the (discounted) ruin probability. This approach goes back to \cite{waters1983some} So if we now assume a constant reinsurance strategy $u$, proportional reinsurance and premiums calculated by the expected value principle as above, equation \eqref{lundberg} becomes
\begin{align*}
 \lambda(\hat{m}_Y(u \gamma)-1)-\delta-(\lambda\beta(1+\eta)-(1-u)(1+\theta)\lambda\beta)\gamma=0.
\end{align*}
Concavity arguments, differentiating and recollecting terms as in \cite{schmidli2007stochastic} now yield the following asymptotically optimal control strategy.
\begin{equation}\label{asyoptstr}
 u^*=\frac{\lambda(\theta-\eta)\left(1-\sqrt{\frac{1}{1+\theta}}\right)}{\delta+2\lambda(1-\sqrt{1+\theta})+\theta\lambda}.
\end{equation}
It is, perhaps, a little surprising that for exponential claims, the optimal strategy does not depend on $\beta$, the expectation of $F_Y$. If we calculate $u^*$ for $\delta=0.05$ and $\lambda=1, \eta=0.5$ and $\theta=0.7$ as above, we get $u^*_{0.05}=0.3275$ which is also indicated by Figure \ref{u_exp_05}.\\
$ $\\
Another very interesting fact is that the asymptotically optimal strategy does not depend on the actual penalty function $w$ as well. This might seem counterintuitive at first, but using material from \cite{asmussen2010ruin}, we see that for a constant strategy $u$
\[
 \lim_{x\to\infty}\Phi^u(x)e^{\gamma(u) x}=C_{\delta,w}.
\]
So only the constant $C_{\delta,w}$ depends on the penalty function $w$, while the asymptotic behaviour is governed by the adjustment coefficient just as in the case of the discounted ruin probability. The reason for this is of course the indicator function for ruin in the Gerber-Shiu function; for high starting values, ruin is just unlikely to occur. Evaluating \ref{asyoptstr} for $\delta=0.1$ yields the asymptotically optimal strategy $u^*_{0.1}=0.2423$ which is confirmed by Figure \ref{u_w2_exp_08}. So for the same values of $\delta$, Figures \ref{u_exp_05} and \ref{u_w2_exp_08} converge to the same level.\\
\vspace{2em}
\centerline{\underline{\hspace*{17cm}}}

\noindent Michael Preischl \Letter\\
Institute for Analysis and Number Theory, Graz University of Technology, Kopernikusgasse 24/II, 8010 Graz, Austria\\
preischl@math.tugraz.at\\
$ $\\
\noindent Stefan Thonhauser\\
Institute for Statistics, Graz University of Technology, Kopernikusgasse 24/III, 8010 Graz, Austria\\
stefan.thonhauser@math.tugraz.at

\bibliography{opt_re_GS.bib}
\bibliographystyle{plainnat}
\end{document}